\documentclass{article}

\usepackage[affil-it]{authblk}

\usepackage[utf8]{inputenc}
\usepackage{amsmath}
\usepackage{amsfonts}
\usepackage{amssymb}
\usepackage{amsthm}
\usepackage{graphicx}
\usepackage[left=2cm,right=2cm,top=2cm,bottom=2cm]{geometry}
\usepackage{enumerate}
\usepackage{appendix}
\usepackage{mathtools}
\usepackage[colorlinks=true,linkcolor=red,citecolor=black,pdfencoding=auto,psdextra]{hyperref}
\usepackage[english]{babel}
\usepackage{tikz-cd}

\usepackage{tikz}
\usetikzlibrary{arrows.meta, positioning}

\newtheorem{theorem}{Theorem}[section]
\newtheorem{lemma}[theorem]{Lemma}

\theoremstyle{definition}

\theoremstyle{remark}
\newtheorem{remark}[theorem]{Remark}

\newcommand{\dd}{\textnormal{d}}

\numberwithin{equation}{section}

\title{Energetic Derivation and Geometric Reduction of Reaction-Diffusion Systems with Holling-Type Functional Responses}
\author{Jan-Eric Sulzbach$\,^{1}$}
\affil{\small
    $\,^1$ Department of Mathematics, Technical University of Munich, Germany\\
    email: janeric.sulzbach@tum.de\\    
}
\date{}

\begin{document}

\maketitle
\begin{abstract}
This paper presents an energetic derivation of a class of multi-species reaction-diffusion systems incorporating various functional responses, with a focus on and application to ecological models. 
Starting point is a closed reaction network for which we apply the Energetic Variational Approach (EnVarA) to derive the corresponding reaction-diffusion system.
This framework captures both diffusion and nonlinear reaction kinetics, and recovers Lotka–Volterra-type dynamics with species-dependent interactions.
Moreover, we find an open subsystem embedded into the larger closed system, that contains the physical and ecological important quantities. 
By analyzing different asymptotic regimes in the reaction parameters of the resulting system, we formally derive classical Holling type I, II, and III functional responses. 
To rigorously justify these reductions and their dynamical properties, we apply the generalized geometric singular perturbation theory (GSPT) for PDEs and prove the existence of an attracting and invariant slow manifolds near the critical manifold of the reduced system. 
Our analysis not only bridges thermodynamic consistency with ecological modeling but also provides a robust framework to study the dynamics of slow manifolds in multi-scale reaction-diffusion systems.
\end{abstract}

\section{Introduction}

Reaction-diffusion systems (RDSs) play a pivotal role in understanding and modeling complex phenomena across biology, chemistry, and ecology. 
In biology, they are instrumental in describing processes such as pattern formation in developmental biology \cite{turing1990chemical,gierer1972theory,britton1986reaction}, the spread of signaling molecules in cellular communication \cite{kondo2010reaction}, the dynamics of population interactions in ecosystems \cite{murray2007mathematical}, and the application to bacterial communication \cite{kuttler2017reaction}.
In chemistry, RDSs are used to model reaction kinetics in spatially distributed systems, such as the propagation of chemical waves in excitable media \cite{epstein1998introduction} and the behavior of auto-catalytic reactions such as the Gray-Scott model \cite{doelman1997pattern,liang2022reversible,hao2024pattern}.
In ecology, these systems help explain the spatial distribution of species \cite{okubo2001diffusion,cosner2008reaction} and vegetation \cite{klausmeier1999regular,rietkerk2002self,meron2012pattern} , the spread of invasive organisms \cite{shigesada1997biological}, and the dynamics of predator-prey interactions \cite{holmes1994partial,sherratt1997oscillations}.
By capturing the interplay between local reactions and spatial diffusion, RDSs provide a powerful framework for analyzing and predicting the behavior of complex systems, making them indispensable tools in both theoretical and applied research.

From a mathematical perspective the study of RDSs dates back to the works of Lotka and Volterra \cite{lotka1920analytical,volterra1926fluctuations} and has been of interest since then as they exhibit a rich variety of structures.
One of the most striking features is pattern formation, where spatially heterogeneous solutions emerge from homogeneous initial conditions, as famously described by the Turing instability \cite{turing1990chemical}. 
These patterns include spots, stripes, and spirals, which are observed in biological systems \cite{gierer1972theory,koch1994biological}, chemical systems like the Belousov-Zhabotinsky reaction \cite{shanks2001modeling,cassani2021belousov} and ecological systems such as dryland ecosystems \cite{klausmeier1999regular,borgogno2009mathematical,rietkerk2008regular}. 
Another key structure are traveling waves, which represent solutions that propagate through space with a constant speed and shape.
Traveling waves are critical in modeling phenomena such as the spread of epidemics, the invasion of species, and the propagation of nerve impulses \cite{fisher1937wave,kolmogorov1937study}. The interplay between nonlinear reaction terms and diffusion gives rise to these complex structures, highlighting the deep connection between mathematical theory and real-world applications.

A deeper understanding of these complex dynamics can be achieved by studying the underlying mathematical equations from a modeling perspective. 
This paper focuses on a general class of reaction-diffusion systems with varying functional responses, derived through different scaling limits.
To obtain this general system, we employ the Energy Variational Approach (EnVarA), a powerful framework that integrates energetic and variational principles to describe complex systems.
The EnVarA systematically derives governing equations by balancing energy dissipation, conservative forces, and external influences \cite{giga2017variational,wang2022some}.
In the context of reaction-diffusion systems (RDSs), this approach naturally incorporates both reaction kinetics and diffusion processes within a unified thermodynamic framework. 
The EnVarA and its extensions have been successfully applied across various models, ranging from fluid dynamics \cite{de2019non,hsieh2020global,liu2022brinkman} to simple non-isothermal reaction-diffusion systems \cite{wang2020field,liang2022reversible,liu2022well}.
A key ingredient in the EnVarA is the free energy, which is a thermodynamic potential that measures the useful work obtainable from a closed thermodynamic system.
Not only is the form of the free energy of RDS appearing in chemical reactions known both from a theoretical point of view \cite{ge2013dissipation} and from an experimental one \cite{beveridge1989free,ensing2006metadynamics,hu2008free},
but the free energy can also be used in biological and ecological systems \cite{jorgensen2004towards,colombo2021first}.
Thus, the EnVarA is a powerful method with versatile applications. 

In this paper, we present an energetic derivation of a multi-species reaction-diffusion system that includes Lotka-Volterra-like terms for the different population sizes. 
Our approach accounts for species-dependent behavior as well as birth and death processes within populations, providing a direct derivation of reaction-diffusion dynamics that explicitly incorporates different functional responses.\\
The term functional response was first introduced by Solomon \cite{solomon1949natural} and later refined by Holling \cite{holling1959components,holling1959some}, who classified functional responses into three fundamental types (Types I, II, and III)—a classification that remains widely used in today's ecological modeling.
For a comprehensive overview of species interactions in ecological systems, we refer to \cite{arditi2012species}.

Starting point for the model is an open reaction-diffusion system satisfying the law of mass action.
The reactions in the system are represented as follows.
\begin{align}\label{eq: reaction 1}
     \mathcal{A} \xlongrightarrow{r} W, \quad W\xlongrightarrow{m_w} \mathcal{A},
\end{align}
where these two reactions can be seen as interactions of the species $W$ with an external environment, which we will call $\mathcal{A}$.
An example is the rain fall at a rate $r$ and the evaporation of water at a rate $m_w$ in a dryland ecosystem.
Next, we have the interactions between two species with their individual environments
\begin{align}\label{eq: reaction 2}
    V_s \xlongrightarrow{m_s} \mathcal{B}_s, \quad  V_h \xlongrightarrow{m_h} \mathcal{B}_h.
\end{align}
This removal from the reaction can be interpreted as the deaths of the species $V_s,\, V_h$ at rates $m_s$ and $m_h$, respectively.
Lastly, we have the interactions between the species given by the following reactions
\begin{align}\label{eq: reaction 3}
    W+V_s\xlongrightarrow{p_1}  2 V_h, \quad    V_h\xlongrightarrow{p_2} V_s.
\end{align}
Note that this is an open system and does not satisfy the mass balance.

In order to make the system thermodynamically consistent we consider the related closed (super-)system  with reversible reactions
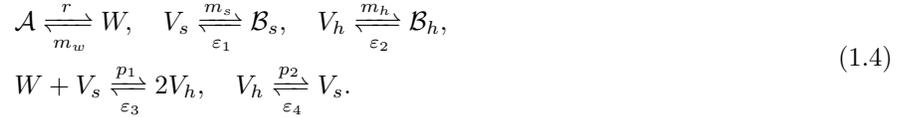
\begin{align}\label{eq: closed reaction}\begin{split}
        &\mathcal{A} \xrightleftharpoons[m_w]{r} W, \quad  V_s \xrightleftharpoons[\varepsilon_1]{m_s} \mathcal{B}_s, \quad  V_h \xrightleftharpoons[\varepsilon_2]{m_h} \mathcal{B}_h,\\
        & W+V_s\xrightleftharpoons[\varepsilon_3]{p_1} 2 V_h, \quad    V_h\xrightleftharpoons[\varepsilon_4]{p_2} V_s.
        \end{split}
\end{align}

\begin{remark}\label{rem: sclaing intro}
    The introduction of the virtual species $\mathcal{A},\,\mathcal{B}_s$ and $\mathcal{B}_h$, representing the atmosphere/hydrosphere and the biosphere respectively in the ecological models we want to study, allows us to close the system.
    Moreover, it is reasonable to assume that the concentrations of these species is large (of order $\mathcal{O}(1/\varepsilon)$) compared to the other species.
    Hence, we can interpret the original system as an open subsystem situated in a sustained environment with influx and efflux from a larger closed universe \cite{ge2013dissipation}.
\end{remark}

In the first part of this paper (see Section \ref{sec: derivation}) we derive the reaction-diffusion system (RDS) corresponding to the above reactions via the energetic variational approach. 
The reaction diffusion system for the concentrations of the species, denoted by lower case letters, and governed by the reaction \eqref{eq: closed reaction} is hence given by
\begin{align}\label{eq: closed reaction system}
    \begin{split}
        \partial_t v_s&= d_{v_s}\Delta v_s- \dot{R}_1 + \dot{R}_2-\dot{R}_3,\\
        \partial_t v_h&= d_{v_h}\Delta v_h + 2 \dot{R}_1-\dot{R}_2 -\dot{R}_4,\\
        \partial_t w&= d_{w}\Delta w- \dot{R}_1 -\dot{R}_5,\\
        \partial_t b_s&= d_{b_s}\Delta b_s+\dot{R}_3,\\
        \partial_t b_h&= d_{b_h}\Delta b_h+\dot{R}_4,\\
        \partial_t a&= d_{a}\Delta a+\dot{R}_5, 
    \end{split}
\end{align}
where the reaction rates $\dot{R}_i$ for the reactions in \eqref{eq: closed reaction} are determined by the law of mass action.
Thus, we have
\begin{align}\label{eq: law of mass action 1}
    \begin{split}
    \dot{R}_1&= p_1 wv_s -\varepsilon_3 v_h^2 ,\\
    \dot{R}_2&= p_2 v_h-\varepsilon_4 v_s ,\\
    \dot{R}_3&= m_s v_s-\varepsilon_1 b_s ,\\
    \dot{R}_4&= m_h v_h -\varepsilon_2 b_h,\\
   \dot{R}_5&= m_w w-ra.\\
    \end{split}
\end{align}
The equations are posed on a bounded domain $\Omega\subset \mathbb{R}^n$ for $n=1,\,2,\,3$ supplemented with non-flux boundary conditions for all species given by
\begin{align}\label{eq: boundary conditions}
    \partial_n v_s=\partial_n v_h=\partial_n w=\partial_n b_s=\partial_n b_h=\partial_n a=0.
\end{align}
In addition, we impose the following initial conditions
\begin{align}\label{eq: initial conditions}
\begin{split}
    v_s(0)=v_{s,0},\, v_h(0)=v_{h,0},\, w(0)=w_0\\
    b_s(0)=b_{s,0},\,b_h(0)=b_{h,0},\, a(0)=a_0,
\end{split}    
\end{align}
satisfying 
\begin{align}
    \int_\Omega v_{s,0}+v_{h,0}+w_0 +b_{s,0}+b_{h,0}+a_0\,\dd x=M_0>0.
\end{align}
Then, the closed system \eqref{eq: closed reaction} satisfies conservation of mass property, i.e.,
\begin{align}\label{eq: cons of mass}
    \int_\Omega \big(v_{s}+v_{h}+w +b_{s}+b_{h}+a\big)(t)\,\dd x=M_0
\end{align}
for any time $t>0$.
Moreover, we can determine the constant stationary solutions $(\bar{v}_s,\bar{v}_h,\bar{w}, \bar{b}_s,\bar{b}_h,\bar{a})$ of the closed and conserved system \eqref{eq: closed reaction system}. 
Using this ansatz we obtain
\begin{align}\label{eq: constant reaction}
\begin{split}
    m_w \bar{w}&= r \bar{a},\quad m_s \bar{v}_s=\varepsilon_1 \bar{b}_s,\quad m_h \bar{v}_h=\varepsilon_2 \bar{b}_h,\\
   & p_1 \bar{w}\bar{v}_s= \varepsilon_3 \bar{v}_h^2 ,\quad p_2 \bar{v}_h=\varepsilon_4 \bar{v}_s, 
\end{split}
\end{align}
together with the mass constraint
\begin{align}\label{eq: mass constraint}
    \bar{v}_s+\bar{v}_h+\bar{w}+ \bar{b}_s+\bar{b}_h+\bar{a}= M_0,
\end{align}
where we assume without loss of generality that $|\Omega|=1$.
Combing \eqref{eq: constant reaction} and \eqref{eq: mass constraint}, we obtain two structurally different kind of constant solutions.
The first one is
 \begin{equation}\label{eq: equilibrium 1}
            \begin{cases}
            \begin{aligned}
                \bar{v}_s &=\frac{M_0}{C}, \quad &\bar{v}_h&= \frac{\varepsilon_4}{p_2}\frac{M_0}{C},\quad &\bar{w}&=\frac{\varepsilon_3\varepsilon_4^2}{p_1p_2^2}\frac{M_0}{C},\\
                \bar{b}_s&=\frac{m_s}{\varepsilon_1}\frac{M_0}{C},\quad &\bar{b}_h&= \frac{m_h\varepsilon_4}{p_2\varepsilon_2}\frac{M_0}{C},\quad &\bar{a}&= \frac{m_w\varepsilon_3\varepsilon_4^2}{rp_1p_2^2}\frac{M_0}{C},
            \end{aligned}
        \end{cases}
    \end{equation}
where $C=1+\frac{\varepsilon_4}{p_2}+\frac{\varepsilon_3\varepsilon_4^2}{p_1p_2^2}+\frac{m_s}{\varepsilon_1}+\frac{m_h\varepsilon_4}{p_2\varepsilon_2}+\frac{m_w\varepsilon_3\varepsilon_4^2}{rp_1p_2^2}$.
The other one is 
\begin{equation}\label{eq: equilibrium 2}
    \begin{cases}
        \begin{aligned}
             \tilde{v}_s &=0, \quad &\tilde{v}_h&= 0,\quad &\tilde{w}&=\frac{r}{m_w+r}M_0,\\
                \tilde{b}_s&=0,\quad &\tilde{b}_h&= 0,\quad &\tilde{a}&=\frac{m_w}{m_w+r}M_0.
        \end{aligned}
    \end{cases}
\end{equation}
In applications one is usually interested in the properties of the first equilibrium, as the second one only consists of two species with non-zero concentration.\\

In the second part of the paper we investigate how different scales of the parameters, which correspond to different time scales in the reactions, can lead to different models via different types of functional responses (see Section \ref{sec: formal reduction}).
Following the classification of Holling \cite{holling1959components,holling1959some} we show first of all how these functional responses can be obtained in a formal way.
Several approaches on justifying and deriving these functional responses have been studied before, such as a quasi-steady state approximation \cite{real1979ecological,segel1989quasi}, asymptotic expansions \cite{dawes2013derivation} or mathematical convergence results \cite{conforto2018reaction,tang2023rigorous}.

However, as the RDSs have intricate and complex dynamics it is key not only to prove the convergence to the reduced systems but also understand how the dynamics of the reduced system relate to the dynamics of the original system.
To overcome this, we work in the context of the geometric singular perturbation theory (GSPT) \cite{fenichel1979geometric,Jones1995,kuehn2015multiple}, which is an approach addressing problems with distinct time scales using a geometric perspective, and which has been recently generalized to systems of partial differential equations \cite{hummel2022slow,kuehn2024infinite,kuehn2025fast} (see Section \ref{sec: gspt ideas}).
The main advantage of the GSPT is that it allows not only the study of the convergence of solutions to the reduced system but also the dynamics of the two systems.
It is based on the work of Fenichel \cite{fenichel1979geometric}, which endeavors to identify the central dynamical structures such as invariant sets and invariant manifolds, present in the phase space of the dynamical system. 
It also seeks to exploit their properties, such as their fast-slow decomposition and the intersections and foliations of various manifolds. 

To conclude this section we apply the GSPT to the to rigorously derive the reduced system with Holling type II functional response \eqref{eq: gspt limit} from the open system \eqref{eq: gspt epsilon} (see Section \ref{sec: gspt applied} and Figure \ref{fig:flow-diagram} below).
\\

Here we summarize the systems and the different limits we consider throughout this paper.\\

\begin{figure}[ht]
\centering
\begin{tikzpicture}[
  node distance=1.5cm and 2cm,
  boxnode/.style={draw, minimum width=3cm, minimum height=1cm, align=center},
  every path/.style={->, >=Stealth}
]

\node[boxnode] (top) {Closed system \eqref{eq: closed reaction}\\ 
internal parameters: reaction rates};
\node[boxnode, below=of top] (middle) {Open system \eqref{eq: reduced system 1}\\ 
internal parameters: reaction rate\\
external parameter: rainfall/ forcing term};
\node[boxnode, below left=of middle] (left) {Holling Type I \eqref{eq: Holling I}};
\node[boxnode, below=of middle] (center) {Holling Type II \eqref{eq: limit holling 2}};
\node[boxnode, below right=of middle] (right) {Holling Type III \eqref{eq: Holling III}};

\draw (top) -- (middle) node[midway, right] {\parbox{3cm}{$\varepsilon_i \to 0$\\ i=1,\dots,7}};
\draw (middle) -- (left) node[midway, left, xshift=-0.2cm] {$p_2\to \infty$};
\draw (middle) -- (center) node[midway, right] {$p_2\to 0$};
\draw (middle) -- (right) node[midway, right, xshift=1.0cm] {\parbox{3cm}{modified reaction\\$p_2\to 0$}};

\end{tikzpicture}
\caption{Scaling Limits and Reduced Models }
\label{fig:flow-diagram}
\end{figure}
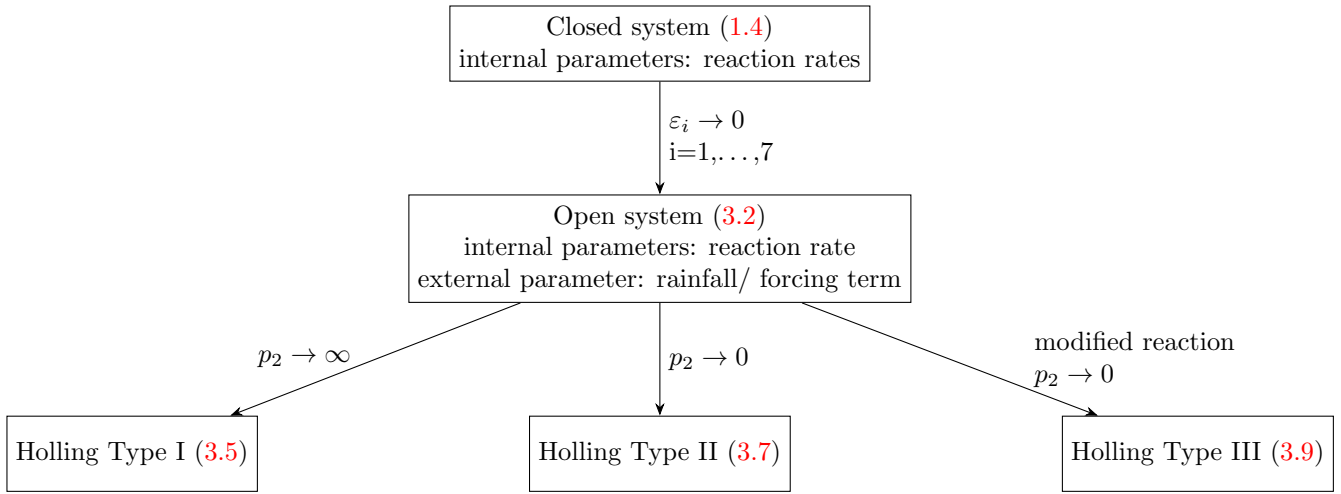

\section{Derivation}\label{sec: derivation}

In this section we derive a general reaction diffusion model with different types of functional responses via the energetic variational approach (EnVarA).
Starting point of the energetic variational approach is an energy-dissipation law for the closed and isothermal system at hand. 
For the derivation of a non-isothermal reaction-diffusion system using a generalized EnVarA based on the first and second law of thermodynamics we refer to \cite{liu2022well}.\\

To keep the notation as concise as possible we set $\textbf{c}=(v_s,v_h,w,b_1,b_2,a)$ and denote its entries by $c_\alpha$.
Then, the energy-dissipation law governing the reaction \eqref{eq: closed reaction} can be written as
\begin{align}\label{eq: energy dissipation law}
    \frac{\dd}{\dd t}\mathcal{F}(v_s,v_h,w,b,a)= \frac{\dd}{\dd t}\mathcal{F}(\mathbf{c})= \mathcal{D}_d+\mathcal{D}_r,
\end{align}
where the free energy of the system is given by 
\begin{align}\label{eq: free energy}
    \mathcal{F}=\int_\Omega \sum_\alpha c_\alpha \big(\ln(c_\alpha/\bar{c}_\alpha)-1\big) \dd x
\end{align}
and $\bar{c}_\alpha$ denote the equilibrium state (see equation \eqref{eq: equilibrium 1}).
The existence of such a Lyapunov functional for a reaction with law of mass action is a well established result see e.g. \cite{ge2016mesoscopic,ge2017mathematical}.

The dissipation consists of two parts coming from the two mechanisms of the reaction diffusion system,
the spatial dissipation given by 
\begin{align}\label{eq: diss dif}
    \mathcal{D}_d= \int_\Omega \sum_\alpha c_\alpha/d_\alpha |\mathbf{u}_\alpha|^2 \dd x,
\end{align}
where $\textbf{u}_\alpha$ is the induced velocity of the diffusion process of each species $c_\alpha$,
and the dissipation due to the chemical reaction is given by
\begin{align}\label{eq: diss reac}
    \mathcal{D}_r= \mathcal{D}_r(\mathbf{R},\dot{\mathbf{R}}),
\end{align}
where we introduce a new state variable, the reaction trajectory $\mathbf{R}$, and its time derivative $\dot{\mathbf{R}}$, the so called reaction rate.

Note that the system satisfies the following kinematics for each species $c_\alpha$
\begin{align}\label{eq: kinematics of reaction}
    \begin{split}
        \partial_t c_\alpha +\nabla(c_\alpha \mathbf{u}_\alpha)&= \sum_{i}\sigma_{i,\alpha} \dot{R}_i ,       
    \end{split}
\end{align}
where $\mathbf{\hat\sigma}$ is the stoichiometric matrix with entries $\sigma_{i,\alpha}$ and where the summation is over each individual reaction.

\subsection{EnVarA for the diffusion}

In a first step we define, for each species $c_\alpha$, a flow map $\mathbf{x}_\alpha(X,t):\Omega \to \Omega$ associated to the velocity $\mathbf{u}_\alpha$, where $X$ denotes the Lagrangian coordinates and $x$ the Eulerian coordinates. 
The flow map for each species $c_\alpha$ is defined via the following ordinary differential equation
\begin{align}
    \frac{\dd}{\dd t} \mathbf{x}_\alpha(X,t) = \mathbf{u}_\alpha(\mathbf{x}_\alpha(X,t),t),\quad \mathbf{x}_\alpha(X,0)=X.
\end{align}
From the least action principle (LAP) it follows that the variation of the action $A$, defined as $A:=\int_0^T \mathcal{F}\dd t$, with respect to the flow map $x$ yields the conservative forces.
To be more precise it holds that
\begin{align}
    \delta_{\mathbf{x}_\alpha} A= \int_0^T \int_\Omega f_{\text{cons}}^\alpha\cdot \delta \mathbf{x}_\alpha\, \dd x\dd t.
\end{align}
In the case of the free energy for the reaction at hand we obtain
\begin{align}
    \delta_{\mathbf{x}_\alpha} A=  \delta \int_0^T \mathcal{F}\dd t =\int_0^T\int_\Omega \nabla c_\alpha \cdot \delta \mathbf{x}_\alpha \, \dd x\dd t.
\end{align}
Next, we consider the dissipation $\mathcal{D}_d$ due to the diffusive contributions and apply the the maximum dissipation principle (MDP), which states that
\begin{align}
    \frac{1}{2}\delta_{\mathbf{u}_\alpha} \mathcal{D}_d= \int_\Omega f_{\text{diss}}^\alpha\cdot \delta \mathbf{u}_\alpha \dd x.
\end{align}
Hence, we obtain that
\begin{align}
    \frac{1}{2}\delta_{\mathbf{u}_\alpha} \mathcal{D}_d=  \int_\Omega  c_\alpha/d_\alpha \mathbf{u}_\alpha \cdot \delta \mathbf{u}_\alpha \dd x.
\end{align}
According to the force balance we have $f_{\text{cons}}^\alpha+f_{\text{diss}}^\alpha=0$, for each species $c_\alpha$ and thus
\begin{align} \label{eq: force balance u}
    \mathbf{u}_\alpha=-d_\alpha \nabla\ln c_\alpha.
\end{align}

\subsection{EnVarA for the reaction}

For the details in the derivation we refer to \cite{wang2020field,liang2022reversible,liu2022well}.
Here, we present an overview of the key concepts.\\

The first step is to replace the concentrations of the species $c_\alpha$, which were useful for determining the mechanical aspects, with the so called reaction trajectory $R_i$.
The reaction trajectory keeps track of each reaction rather than the individual species. 
This method goes back to the work of \cite{perelson1974chemical,oster1974chemical} and has its origins in the work of \cite{de1936thermodynamic}.
Hence, using the mass constraint we can interchange the concentrations with the reaction trajectories of each of the reactions in \eqref{eq: reaction 3} and reduce the dimension of the system by one.
The reaction velocity or reaction rate is denoted by $\dot R_i$ and defined via $\frac{\dd}{\dd t} R_i=:\dot{R}_i$, for $i=1,\dots,5$.
Then, the concentration can be written as 
\begin{align}\label{eq: reaction trajectories}
       \textbf{c}(t)= \textbf{c}(0) +\mathbf{\hat\sigma} \textbf{R}, 
\end{align}
which we can regard as the kinematics for the chemical reaction.

Now, we can rewrite the free energy in terms of the reaction trajectories, i.e., $\mathcal{F}(\textbf{R})=\mathcal{F}(\textbf{c}(\textbf{R}))$.
Similarly, we can rewrite the part of the dissipation due to the reaction in terms of $\textbf{R}$ and $\dot{\textbf{R}}$.
For the chemical part we then have 
\begin{align}\label{eq: engery disspation reaction part}
    \frac{\dd}{\dd t} \mathcal{F}(\textbf{R})= - \mathcal{D}_r(\textbf{R},\dot{\textbf{R}}).
\end{align}
As the reaction terms are usually far from equilibrium one does not expect a quadratic dependence in terms of $\dot{\textbf{R}}$.
Therefore, we assume that we can write the dissipation as
\begin{align*}
    \mathcal{D}_r(\textbf{R},\dot{\textbf{R}})= \int_\Omega \mathcal{G}_r(\textbf{R},\dot{\textbf{R}})\cdot \dot{\textbf{R}}\,\dd x . 
\end{align*}
Combining this with the observation that
\begin{align*}
    \frac{\dd}{\dd t} \mathcal{F}(\textbf{R})=\int_\Omega \delta_{\textbf{R}} \mathcal{F}\cdot \dot{\textbf{R}} \,\dd x
\end{align*}
yields
\begin{align*}
    \mathcal{G}_r(\textbf{R},\dot{\textbf{R}})=-\delta_\textbf{R} \mathcal{F},
\end{align*}
which can be interpreted as a general gradient flow.
A direct computation shows 
\begin{align}\label{eq: variation reaction rate}
   \delta_{R_i} \mathcal{F}(\textbf{R}) =\sum_{\alpha} \ln \frac{c_\alpha}{\bar c_\alpha}\frac{\partial c_\alpha}{\partial R_i}.
\end{align}
We want to point out that this term is also known as the affinity of the reaction \cite{de1936thermodynamic}.
In a sense the affinity plays the role of a "force" that drives the chemical reaction, see \cite{de1936thermodynamic,kondepudi2014modern} for more details.
A general form for the diffusion of the reaction part is given by
\begin{align}\label{eq: detailed diss}
    \mathcal{D}_r(\textbf{R},\dot{\textbf{R}})= \sum_i \dot{\textbf{R}}_i \ln\bigg(\frac{\dot{\textbf{R}}_i}{\eta_i(\textbf{c}(\textbf{R}))}+1\bigg).
\end{align}
In particular, we set $\eta_i(\textbf{c}(\textbf{R}))= R_i^-$ to be consistent with the mass action kinetics.
Here, $ R_i^-$ denotes the part of the reaction trajectory $R_i$ that accounts for the back reaction.
Thus, for the first term in the diffusion of the reaction part we obtain
\begin{align}\label{eq: diffusion of reaction 1}
   \mathcal{D}_{r_1}(R_1,\dot R_1)= \dot R_1 \ln\bigg(\frac{\dot R_1}{\varepsilon_3 v_h^2}+1\bigg),
\end{align}
and similar for the other terms.
Then, combining equations \eqref{eq: variation reaction rate} and \eqref{eq: detailed diss} via the dissipation law \eqref{eq: engery disspation reaction part} and assuming that the law of mass action holds yields
\begin{align}\label{eq: reaction rates}
\begin{split}
    \dot{R}_1&= p_1 wv_s -\varepsilon_3 v_h^2 ,\\
    \dot{R}_2&= p_2 v_h-\varepsilon_4 v_s ,\\
    \dot{R}_3&= m_s v_s-\varepsilon_1 b_s ,\\
    \dot{R}_4&= m_h v_h -\varepsilon_2 b_h,\\
   \dot{R}_5&= m_w w-ra.\\
\end{split}
\end{align}

Then, combining the contribution of the diffusion \eqref{eq: force balance u} and reaction process \eqref{eq: reaction rates} with the kinematics of the system  \eqref{eq: kinematics of reaction} yields exactly the reaction diffusion system \eqref{eq: closed reaction system} given in the introduction.

Hence, we have shown how to derive a reaction-diffusion system from energetic considerations via prior knowledge/ assumptions on the free energy, the dissipation and the reaction kinematics, i.e. in this case the law of mass action kinematics.

\section{Functional Responses and Reduced Models}\label{sec: formal reduction}
The main result of this section  is to show how different scales in the parameters lead to different functional responses in the reduced models.\\

As the concentrations of the species $b_s,\, b_h$ and $a$  are much greater than the concentrations of the other species, we can introduce a parameter $\varepsilon$ to denote the different scales as indicated in Remark \ref{rem: sclaing intro}.
Moreover, this assumption allows us to neglect the influence of diffusion in these equations.
Thus, the rescaled system now reads as
\begin{align}\label{eq: rescaled system full}
\begin{split}
    \partial_t v_s^\varepsilon&= d_{v_s}\Delta v_s^\varepsilon- p_1w^\varepsilon v_s^\varepsilon +\varepsilon_3(v_h^\varepsilon)^2 -m_s v_s^\varepsilon+ p_2 v_h^\varepsilon -\varepsilon_4 v_s^\varepsilon +\varepsilon_1 b_s^\varepsilon,\\
    \partial_t v_h^\varepsilon&= d_{v_h}\Delta v_h^\varepsilon + 2 p_1w^\varepsilon v_s^\varepsilon-2\varepsilon_3 (v_h^\varepsilon)^2 -m_h v_h^\varepsilon -p_2 v_h^\varepsilon  +\varepsilon_4 v_s^\varepsilon +\varepsilon_2 b_h^\varepsilon,\\
    \partial_t w^\varepsilon&= d_{w}\Delta w^\varepsilon- p_1w^\varepsilon v_s^\varepsilon +\varepsilon_3(v_h^\varepsilon)^2 -m_w w^\varepsilon+ ra^\varepsilon,\\ 
    \partial_t b_s^\varepsilon&= \varepsilon_5\big( m_s v_s^\varepsilon-\varepsilon_1 b_s^\varepsilon\big) ,\\
    \partial_t b_h^\varepsilon&=\varepsilon_6\big(m_h v_h^\varepsilon -\varepsilon_2 b_h^\varepsilon \big) ,\\
    \partial_t a^\varepsilon&= \varepsilon_7\big(m_w w^\varepsilon-ra^\varepsilon\big) .
    \end{split}
\end{align}

In addition, we assume from now on that $\varepsilon_j=\varepsilon$ for all $j=1,\dots, 7$.
Then we can reduce the complexity of the system by considering the formal limit as $\varepsilon\to 0$ and obtain
\begin{align} \label{eq: reduced system 1}
    \begin{split}
        \partial_t v_s^0&= d_{v_s}\Delta v_s^0- p_1w^0 v_s^0  -m_s v_s^0+ p_2 v_h^0,\\
    \partial_t v_h^0&= d_{v_h}\Delta v_h^0 + 2 p_1w^0 v_s^0 -m_h v_h^0 -p_2 v_h^0,\\
    \partial_t w^0&= d_{w}\Delta w^0- p_1w^0 v_s^0 -m_w w^0+ r a_0,\\
    \end{split}
\end{align}
subject to the boundary and initial conditions of \eqref{eq: boundary conditions}-\eqref{eq: initial conditions}. 
\begin{remark}
    Here, we briefly motivate how to obtain the formal limit.
    As the last three equations in system \eqref{eq: rescaled system full} are linear we have
    \begin{align}\label{eq: var of const for a}
        a(t)= a_0\text{e}^{-\varepsilon r t} +\int_0^t \text{e}^{-\varepsilon r (t-s)} m_w w^\varepsilon(s)\dd s
    \end{align}
    and similar for $b_s(t)$ and $b_h(t)$. 
    Assuming that $v_s^\varepsilon,v_h^\varepsilon,w^\varepsilon\in L^\infty(\Omega\times (0,T))$, which can be shown rigorously, see e.g. \cite{pierre2010global}, it follows that 
    \begin{align*}
        \varepsilon (v_h^\varepsilon)^2,\, \varepsilon v_s^\varepsilon,\, \int_0^t \text{e}^{-\varepsilon r (t-s)} \varepsilon m_w w^\varepsilon(s)\dd s,\, \int_0^t \text{e}^{-\varepsilon r (t-s)} \varepsilon m_h v_h^\varepsilon(s)\dd s,\, \int_0^t \text{e}^{-\varepsilon r (t-s)} \varepsilon m_s v_s^\varepsilon(s)\dd s \to 0
    \end{align*}
    and $a_0\text{e}^{-\varepsilon r t} \to a_0$ as $\varepsilon\to 0$.
    Thus, plugging equation \eqref{eq: var of const for a} into \eqref{eq: rescaled system full} and passing to the limit yields the desired result.

    For more details of this reduction and the convergence of solutions we refer to \cite{liang2022reversible}
\end{remark}

In the next steps we want to further reduce the system by using  different scaling limits of the parameters.

\subsection{Type I functional response}

We assume that reaction rate of the switching between the species $v_h$ and $v_s$ satisfies $p_2\gg 1$.
This corresponds to the case of an almost instant transition from $v_h$ to $v_s$.
Setting $p_2=\varepsilon^{-1}\tilde{p}_2$ we have
\begin{align}
    \begin{split}
         \partial_t v_s&= d_{v_s}\Delta v_s- p_1wv_s -m_s v_s+ \frac{1}{\varepsilon}\tilde{p}_2 v_h ,\\
        \partial_t v_h&= d_{v_h}\Delta v_h + 2p_1wv_s -m_h v_h -\frac{1}{\varepsilon}\tilde{p}_2 v_h ,\\
        \partial_t w&= d_{w}\Delta w-p_1 wv_s -m_w w+ ra_0.\\
    \end{split}  
\end{align}
The limit system as $\varepsilon\to 0$ is given by
\begin{align}\label{eq: Holling I}
    \begin{split}
        \partial_t v&= d_v\Delta v + p_1 wv -m_v v,\\
        \partial_t w&= d_{w}\Delta w- p_1 wv -m_w w+ ra_0,
    \end{split}
\end{align}
where $v=v_s+v_h$.
This yields a functional response of Holling type I in both the $v$ and $w$-component, i.e., the functional response is given by $f(x)= c x$.

\begin{remark}
    In order to obtain a well-known reaction diffusion system from ecology, the Klausmeier model for dryland ecosystems, we have to adjust the reactions between the three species \eqref{eq: reaction 3} as follows
    \begin{align*}
        W+2V_s\xlongrightarrow{p_1}  3 V_h, \quad    V_h\xlongrightarrow{p_2} V_s.
    \end{align*}
    Then, with the same scaling of $p_2$ as before we obtain the reaction diffusion system
    \begin{align*}
         \partial_t v&= d_v\Delta v + p_1 wv^2 -m_v v,\\
        \partial_t w&= d_{w}\Delta w- p_1 wv^2 -m_w w+ ra_0,
    \end{align*}
\end{remark}
which has a Holling type I functional response in the $w$-component and a quadratic response in the $v$-component.
\subsection{Type II functional response}

In this section we consider the regime when the reaction rate $p_2\ll 1$, which is equivalent to considering the case  $p_1\gg 1$.
This corresponds to the case that the production of the species $v_h$ from $w$ and $v_s$ is much faster than the switch from species $v_h$ to $v_s$.
Hence, we introduce a small parameter $\varepsilon$ such that $p_1=\varepsilon^{-1}\tilde{p}_1$. 
In addition, we have to rescale the concentration $v_s$ by $v_s=\varepsilon \tilde{v}_s$ and also the death rate $m_s$ by $m_s=\varepsilon^{-1}\tilde{m}_s$ as they occur on a faster time scale.
Note, that for now, we do not assume that the diffusion coefficient $d_{v_s}$ is on the faster time scale.
Hence, we obtain the following system
\begin{align}\label{eq: rescaled holling 2}
    \begin{split}
         \varepsilon \partial_t \tilde{v}_s&= \varepsilon     d_{v_s}\Delta \tilde{v}_s- \tilde{p}_1w\tilde{v}_s -\tilde{m}_s\tilde{v}_s+ p_2 v_h ,\\
        \partial_t v_h&= d_{v_h}\Delta v_h + 2\tilde{p}_1 w\tilde{v}_s -m_h v_h -p_2 v_h ,\\
        \partial_t w&= d_{w}\Delta w- \tilde{p}_1 w\tilde{v}_s -m_w w+ a.
    \end{split}
\end{align}
To obtain the formal limit as $\varepsilon \to 0$ we consider an asymptotic expansion in $\varepsilon$, which yields in particular that
\begin{align*}
    \tilde{v}_s =\frac{p_2 v_h}{\tilde{p}_1 w+\tilde{m}_s} +\mathcal{O}(\varepsilon).
\end{align*}
Plugging this into \eqref{eq: rescaled holling 2} and letting $\varepsilon\to 0$ we obtain
\begin{align}\label{eq: limit holling 2}
    \begin{split}
        \partial_t v&= d_{v}\Delta v + \frac{2\tilde{p}_1p_2 v w  }{\tilde{p}_1 w+\tilde{m}_s} -m_v v -p_2 v ,\\
        \partial_t w&= d_{w}\Delta w-\frac{\tilde{p}_1 p_2 v w  }{\tilde{p}_1 w+\tilde{m}_s}  -m_w w+ a.
    \end{split}
\end{align}
This then yields the a functional response of Holling type II in the $w$-component, i.e., the functional response is given by the form $f(x)= xy/(x+C)$.
For a rigorous derivation we refer to Section \ref{sec: gspt reduction}.

\begin{remark}
    In the case when the diffusion coefficient $d_{v_s}$ is also on the faster time scale, i.e., we have $d_{v_s}= \varepsilon^{-1} \tilde d_{v_s}$.
    Then the first equation in \eqref{eq: rescaled holling 2} reads as
    \begin{align*}
        \varepsilon \partial_t \tilde{v}_s&=   \tilde d_{v_s}\Delta \tilde{v}_s- \tilde{p}_1w\tilde{v}_s -\tilde{m}_s\tilde{v}_s+ p_2 v_h.
    \end{align*}
    In this case we have
    \begin{align*}
        \tilde v_s= \big(-\tilde d_{v_s} \Delta + \tilde{p}_1w + \tilde{m}_s\big)^{-1} p_2 v_h.
    \end{align*}
    This expression is well-defined as the eigenvalues of the operator $A:=-\tilde d_{v_s} \Delta + \tilde{p}_1w + \tilde{m}_s $ are positive and bounded away from zero, for any relevant functions $w$.
    Thus, the inverse operator $A^{-1}$ is well-defined.

    For the reduced system as $\varepsilon\to 0$ we then obtain 
    \begin{align*}
        \partial_t v&= d_{v}\Delta v + 2\tilde{p}_1p_2  w A^{-1} v  -m_v v -p_2 v ,\\
        \partial_t w&= d_{w}\Delta w- \tilde{p}_1 p_2 w A^{-1} v -m_w w+ a.
    \end{align*}
\end{remark}

\subsection{Type III functional response}

In this section we briefly present how to generate a type III functional response by modifying the underlying individual reactions.
We replace the reactions \eqref{eq: reaction 3} by
\begin{align}
     V_s+k W\xlongrightarrow{p_1}(k+1) V_h, \quad    V_h\xlongrightarrow{p_2} V_s
\end{align}
for some $k\in \mathbb{N}, k>1$.
With this we can proceed as in the previous steps to obtain the reaction diffusion system. Moreover, assuming that we are in the same regime as for the type II functional response we can reduce the equations and obtain
\begin{align}\label{eq: Holling III}
    \begin{split}
         \partial_t v_h&= d_{v_h}\Delta v_h + \frac{(k+1)\tilde{p}_1 p_2 v_h w^k  }{\tilde{p}_1 w^k +\tilde{m}_s} -m_h v_h -p_2 v_h ,\\
        \partial_t w&= d_{w}\Delta w-\frac{k \tilde{p}_1 p_2 v_hw^k  }{ \tilde{p}_1 w^k +\tilde{m}_s}  -m_w w+ a,
    \end{split}
\end{align}
which corresponds to a Holling type III functional response. 

\section{GSPT Reduction}\label{sec: gspt reduction}

In this section we give an overview and show how we can apply the generalized geometric singular perturbation theory for PDEs to make the reductions from the previous section rigorous, focusing on the derivation of the Holling type II functional response.
The advantage of the GSPT is that it not only provides a quantitative estimate on the convergence of solutions of systems but also provides comparative results on the dynamics of these solutions.\\

\subsection{Basic set-up and key ideas}\label{sec: gspt ideas}

The general equations we consider are singularly perturbed systems of semi-linear PDEs with two distinct time-scales, of the form
\begin{align} \label{eq: syst 1}
    \begin{split}
        \varepsilon_1 \partial_t u&= \varepsilon_2 \Delta u +\varepsilon_3 f(u,v),\\
        \partial_t v &= d\Delta v+ g(u,v).
    \end{split}
\end{align}
Here, $\varepsilon_1,\varepsilon_2,\varepsilon_3$ are the parameters relating to the different time-scales, and whenever one of them is small this yields the singular character of the system.
There are several possibilities of the scaling
\begin{enumerate}[(i)]
    \item The case $0<\varepsilon_1=\varepsilon\ll 1$, $\varepsilon_2=\varepsilon_3=\mathcal{O}(1)$ has been studied in the abstract Banach space framework in \cite{hummel2022slow};
    \item The case $0<\varepsilon_1=\varepsilon_2=\varepsilon\ll 1$, $\varepsilon_3=\mathcal{O}(1)$, also known as fast-reaction case, has been studied for the linear case in \cite{kuehn2024infinite} and for the general abstract setting in \cite{kuehn2025fast}. This is also the set-up for the reduction considered in this section going forward;
    \item The case $0<\varepsilon_1=\varepsilon_3=\varepsilon\ll 1$, $\varepsilon_2=\mathcal{O}(1)$ has been studied in \cite{kuehn2025approximate} with an application to the Fokker-Planck equation.
\end{enumerate}
The nonlinear functions $f,g$ are assumed to be sufficiently smooth.
Then, taking the limit $\varepsilon_1\to 0$ we obtain teh reduced system
\begin{align}\label{eq: syst 2}
    \begin{split}
        0&= f(u,v),\\
        \partial_t v &= d\Delta v+ g(u,v).
    \end{split}
\end{align}
This is an algebraic-differential system that describes the evolution of the slow variable constrained to the set $f(u,v)=0$.

The goal of the singular perturbation theory is to understand the dynamics of \eqref{eq: syst 1} by looking at it as a perturbation of the reduced system \eqref{eq: syst 2}.
To be more precise, let $f$ satisfy the conditions of the implicit function theorem in Banach spaces, where in addition we assume that the spectrum of linear operator $D_u f(u,v)$ is contained in the left half plane and is uniformly bounded away from the imaginary axis.
Then we can define an invariant and attracting normally hyperbolic manifold 
\begin{align}\label{eq: critical manifold 1}
    M_0=\{(u,v):~ f(u,v)=0\}= \{v:~ (h(v),v)\},
\end{align}
where the function $h$, coming from the implicit function theorem, satisfies $h(v)=u$ and $f(h(v),v)=0$.\\
In addition, to better understand the fast transition to the semi-flow on the slow manifold we introduce a new time scale via the following transformation $t=\varepsilon \tau$. 
We call $\tau$ the fast time scale and whereas $t$ is referred to as slow time,
Thus we can rewrite \eqref{eq: syst 1} as follows
\begin{align}\label{eq: syst 3}
    \begin{split}
        \partial_\tau u&= \varepsilon \Delta u+ f(u,v),\\
        \partial_\tau v&= \varepsilon \big(d\Delta v+ g(u,v)\big).
    \end{split}
\end{align}
Here, we observe a problem that is not present in the finite dimensional setting of ODEs.
Note, that on the fast time scale we can never view $\Delta v$ as a bounded perturbation as $\Delta$ is an unbounded differential operator.
Hence, we would encounter a limit of the form "$0\cdot\infty$" and furthermore have that $\Delta v$ is not necessarily “small” in any norm compared to the linear part of the $u$-variable.
To overcome this we introduce a suitable splitting of the slow variable space, depending on a small parameter $\zeta$, into a truly slow part and a part with fast dynamics.

The slow part comes from modes/directions, where the Laplacian yields a sufficiently small bounded
perturbation so that these modes are slow. Moreover, the linear part of the dynamics
on this subspace is supposed to exist also backwards in time. 
The other subspace contains the modes, which are fast as $\Delta$ dominates the small parameter $\varepsilon$. 
The parameter $\zeta$ describes which parts of the linear dynamics in the slow variable space are considered as fast and which ones are considered as slow.
In addition, this parameter relates to the size of spectral gaps in the spectrum of the Laplacian $\Delta$.

Now, we have all the basics together to state the generalized Fenichel’s theorem for semi-linear PDEs.
\begin{theorem}\label{thm: general fenichel}
Suppose $M_0= \{ (u,v):~f(u,v)=0\}$ an attracting normally hyperbolic manifold and suppose that $f$ and $g$ are smooth nonlinear functions.
In addition let $\zeta>0$ be the small parameter inducing the splitting in the slow variable space.
Then for sufficiently small $0<\varepsilon \ll 1$, satisfying $\varepsilon \zeta^{-1}<1$  there exists an invariant and exponentially attracting manifold $M_{\varepsilon,\zeta}$ to system \eqref{eq: syst 1}, called slow manifold, which is $\mathcal{O}(\varepsilon)$ close and diffeomorphic to $M_0$.
Moreover, the restriction of the semi-flow of \eqref{eq: syst 1} to $M_{\varepsilon,\zeta}$ is a small perturbation of the semi-flow of the reduced system \eqref{eq: syst 2}.    
\end{theorem}

As by \eqref{eq: critical manifold 1} the critical manifold $M_0$ can be expressed as a graph of a function $h$ the above theorem implies that the slow manifold can obtained by the graph of a perturbed function $h^{\varepsilon,\zeta}$, i.e. 
\begin{align*}
    M_{\varepsilon,\zeta}=\{ v_s:~h^{\varepsilon,\zeta}(v_s)=(u,v_f)\},
\end{align*}
where $v_s$ denotes a function from the slow subspace and $v_f$ a function from the fast subspace.
Moreover, we obtain that the semi-flow on the slow manifold $M_{\varepsilon,\zeta}$ is given by
\begin{align}
    \partial_t v_s= d\Delta v_s+ \text{P}_s g(h^{\varepsilon,\zeta}_u(v_s),h^{\varepsilon,\zeta}_f(v_s),v_s),
\end{align}
where $\text{P}_s$ denotes the projection onto the slow subspace.
For more precise statements we refer to \cite{kuehn2025fast}.

\subsection{Application of the GSPT}\label{sec: gspt applied}

In this part, as an example of the abstract result, we apply Theorem \ref{thm: general fenichel} to the derivation of the Holling type II functional response.

Let $\Omega\subset \mathbb{R}^n$ be an open and bounded domain with sufficiently smooth boundary and in what follows we set $n=2$.
Then, we consider the two systems 
\begin{align}\label{eq: gspt epsilon}
    \begin{split}
         \varepsilon \partial_t \tilde{v}_s^\varepsilon&= \varepsilon     d_{v_s}\Delta \tilde{v}_s^\varepsilon- \tilde{p}_1w^\varepsilon\tilde{v}_s^\varepsilon -\tilde{m}_s\tilde{v}_s^\varepsilon+ p_2 v_h^\varepsilon ,\\
        \partial_t v_h^\varepsilon&= d_{v_h}\Delta v_h^\varepsilon + 2\tilde{p}_1 w^\varepsilon \tilde v_s^\varepsilon -m_h v_h^\varepsilon -p_2 v_h^\varepsilon ,\\
        \partial_t w^\varepsilon &= d_{w}\Delta w^\varepsilon - \tilde{p}_1 w^\varepsilon \tilde{v}_s^\varepsilon -m_w w^\varepsilon+ a,\\
        \partial_n \tilde{v}_s^\varepsilon&=0, \quad \partial_n v_h^\varepsilon=0,\quad \partial_n w^\varepsilon=0,\\
        \tilde{v}_s^\varepsilon(0)&= \tilde{v}_{s,0}^\varepsilon,\quad v^\varepsilon_{h}(0)= v^\varepsilon_{h,0},\quad  w^\varepsilon(0)= w^\varepsilon_0
    \end{split}
\end{align}
and
\begin{align}\label{eq: gspt limit}
    \begin{split}
         \tilde{v}_s &=\frac{p_2 v_h}{\tilde{p}_1 w+\tilde{m}_s} \\
        \partial_t v_h&= d_{v_h}\Delta v_h + 2\tilde{p}_1 w\tilde{v}_s -m_h v_h -p_2 v_h ,\\
        \partial_t w&= d_{w}\Delta w- \tilde{p}_1 w\tilde{v}_s -m_w w+ a,\\
        \partial_n \tilde{v}_s&=0, \quad\partial_n v_h=0,\quad \partial_n w=0,\\
        \tilde{v}_s(0)&= \tilde{v}_{s,0},\quad v_h(0)=v_{h,0},\quad w(0)=w_0,
    \end{split}
\end{align}
where we require that the initial data satisfies 
\begin{align}\label{eq: initial data requirement}
    \tilde{v}_{s,0} =\frac{p_2 v_{h,0}}{\tilde{p}_1 w_0+\tilde{m}_s}.
\end{align}

In a first step we state the result concerning the existence of global classical solutions to systems \eqref{eq: gspt epsilon} and \eqref{eq: gspt limit}. 
\begin{lemma}\label{lem: existence solutions}
\hfill
\begin{enumerate}[a)]
    \item Let $\tilde{v}_{s,0}^\varepsilon,v_{h,0}^\varepsilon,w_0^\varepsilon\in L^\infty(\Omega)$. Then, for any $\varepsilon>0$ there exists a global classical solution to \eqref{eq: gspt epsilon}.
    \item Let $\tilde{v}_{s,0},v_{h,0},w_0\in L^\infty(\Omega)$ and let \eqref{eq: initial data requirement} hold. Then, there exists a global classical solution to \eqref{eq: gspt limit}.
\end{enumerate}
\end{lemma}
\begin{proof}
    The proof of both parts of the lemma follows along the same lines. 
The key insight is that both systems satisfy the so called quasi-positivity condition and a generalized mass constraint. 
We refer to \cite[Thm. 1]{pierre2010global} for more details.
\end{proof}
In particular from the classical maximum/minimum principle \cite{evans2010partial} it then follows that, if the initial data is non-negative, the solution remains non-negative as well.

\begin{remark}
    Besides establishing the existence of solutions, the above results also provides an important estimate that is important for the next step in the GSPT.
To be more precise, we have that the second derivatives of the solutions are bounded, which is crucial for the next steps, where we aim to apply the results of \cite{kuehn2025fast} to the system at hand.
\end{remark}

Now, a natural function space to consider this problem is $H^2(\Omega)$, which is defined as  $H^2(\Omega)=\{u\in L^2:~ \partial^\alpha f\in L^2(\Omega)~~\text{for all } \alpha\in \mathbb{N}^2,~0\leq |\alpha|\leq 2\}$.
In the next step we have to verify several assumptions which guarantee the convergence of solutions of system \eqref{eq: gspt epsilon} to solutions of the limiting system $\eqref{eq: gspt limit}$.\\
First, we observe that the Laplacian with zero Neumann boundary data generates a $C_0$-semigroup satisfying the growth bound
\begin{align*}
    \|\text{e}^{\Delta t}\|_{\mathcal{B}(H^\alpha,H^2)}\leq C t^{\alpha-1 }\text{e}^{\omega_\Delta t},
\end{align*}
for all $t\in (0,\infty)$, where $\omega_\Delta\leq 0$.\\
Next, we have that Lemma \ref{lem: existence solutions} implies that $\Delta\tilde v_s^\varepsilon, \Delta v^\varepsilon_h,\Delta w^\varepsilon \in L^\infty(\Omega)$.
Hence, there exist constants $L_f,L_{g_1}, L_{g_2}$ such that the nonlinear functions
\begin{align*}
    f&: H^2(\Omega)^3 \to H^2(\Omega),~ f(v_s,v_h,w)= -\tilde{p}_1w v_s -\tilde{m}_s v_s+ p_2 v_h,\\
    g_1&:  H^2(\Omega)^3 \to H^2(\Omega),~g_1(v_s,v_h,w)= 2\tilde{p}_1w v_s -\tilde{m}_h v_h- p_2 v_h,\\
    g_2&:  H^2(\Omega)^3 \to H^2(\Omega),~g_2(v_s,v_h,w)= -\tilde{p}_1w v_s -\tilde{m}_w w +a,
\end{align*}
satisfy
\begin{align*}
    \|\text{D} f(v_s,v_h,w)\|_{\mathcal{B}(H^2(\Omega)^3,H^2(\Omega))}&\leq L_f,\\
     \|\text{D} g_1(v_s,v_h,w)\|_{\mathcal{B}(H^2(\Omega)^3,H^2(\Omega))}&\leq L_{g_1},\\
      \|\text{D} g_2(v_s,v_h,w)\|_{\mathcal{B}(H^2(\Omega)^3,H^2(\Omega))}&\leq L_{g_2}.
\end{align*}
In particular. we note that $D_{v_s} f(v_s,v_h,w)=-(p_1w +\tilde{m}_s)\leq -\tilde m_s$, by the non-negativity of $w$.
Hence, the function $f$ satisfies all assumption from the implicit function theorem.
Thus there exists a function $h^0: H^2(\Omega)^2 \to H^2(\Omega)$ such that $h^0(v_h,w)=\frac{p_2 v_h}{\tilde p_1 w+\tilde m_s}= v_s$ and $f(h(v_h,w),v_h,w)=0$ for all $v_h,w$ in $H^2(\Omega)$ such that $v_h,w \geq 0$.
Then, we can view \eqref{eq: gspt limit} as a system formulated on the critical manifold $S_0$, defined as
\begin{align}
    S_0=\bigg\{ (\tilde v_s,v_h,w)\in H^2(\Omega):~ \tilde v_s = h^0(v_h,w)= \frac{p_2 v_h}{\tilde{p}_1 w+\tilde{m}_s}\bigg\}.
\end{align}
Moreover, we can compute the Lipschitz constant of $h^0$ explicitly, i.e.
\begin{align*}
  \|h^0(v_1,w_1)-h^0(v_2,w_2)\|_{H^2}\leq L_{h^0}\big(\| v_1-v_2\|_{H^2}+\|w_1-w_2\|_{H^2}\big),
\end{align*}
where $v_1,v_2,w_1,w_2\in H^2(\Omega)$ and non-negative.
This simplifies the assumptions in \cite[Sec.2]{kuehn2025fast}.

Then we obtain the following quantitative convergence result.
\begin{lemma}\label{lem: convergence of sol}
    There are constants $C_1,\,C_2>0$ independent of $\varepsilon$ such that the solution $(\tilde{v}_{s}^\varepsilon,v_{h}^\varepsilon,w^\varepsilon)$ of system \eqref{eq: gspt epsilon} and  $(v_s,v_h,w)$ of system \eqref{eq: gspt limit}  satisfy
\begin{align*}
     \bigg\| \begin{pmatrix} v_s^\varepsilon(t)-v_s(t)\\ v_h^\varepsilon(t)-v_h(t)\\ w^\varepsilon(t)-w(t) \end{pmatrix} \bigg \|_{H^1\times H^1\times H^1}&\leq C_1 \textnormal{e}^{\varepsilon^{-1} t}\|v_{s,0}^\varepsilon- h^0(v_{h,0},w_0)\|_{H^2}\\
      &\quad+  C_2 \varepsilon^{1/2}  \big(\|v_{h,0}^\varepsilon - v_{h,0}\|_{H^2} + \|v_0\|_{H^2} +\|w^\varepsilon_0-w_0\|_{H^2} +\|w_0\|_{H^2}\big)
\end{align*}
for all $t\in [0,\infty)$ and all $\varepsilon>0$.
\end{lemma}

\begin{remark}
    We consider the convergence only in $H^1$ in order to avoid technical difficulties that arise in the higher order estimates.
\end{remark}

To apply the generalized Fenichel theorem for PDEs we have to introduce a splitting of the $L^2$-space in order to account for the fast modes of the Laplacian in the slow variable.
To this end let $0<\zeta$ be a small parameter.
For simplicity we assume that $\Omega=(0,\pi)^2$.
Then, we can introduce the splitting $L^2(\Omega)= L^2_s(\Omega)\oplus L^2_f(\Omega)$, where the slow component is given by
\begin{align*}
    L^2_s(\Omega)=\text{span}\{x\mapsto \exp(ikx):~ |k|\leq k_0,~k\in \mathbb{Z}\}
\end{align*}
and $k_0\in \mathbb{N}$ is chosen such that $k_0^2\leq \tilde{m}_s \zeta^{-1}\leq (k_0+1)^2$.
Moreover, we denote by $\text{P}_s$ the projection onto the slow space, which by definition of the slow space commutes with the Laplace operator.
And a similar representation holds for $L^2_f(\Omega)$.
In particular, this splitting utilizes the spectral gap in the Laplacian, i.e. we can always find two eigenvalues such that their distance is at least $\mathcal{O}(\log(\zeta^{-1}))$.
The connection between $\zeta$, the parameter of the splitting and $\varepsilon$, the parameter of the timescale separation, is that we assume that $0< \varepsilon\zeta^{-1}<c\leq 1$.
Lastly, we have to assume that $p_2< \tilde{m}_s$, which we use in order to show that a generalized spectral gap condition holds.

With all the assumption in \cite[Sec. 3]{kuehn2025fast} verified, we can now apply the main result for the existence of a slow manifold.
\begin{theorem}\label{thm: gspt example}
Suppose $S_{0,\zeta}$ is a submanifold of the critical manifold $S_{0}$ given by
\begin{align*}
    S_{0,\zeta}= \{(v_h,w)\in S_0: \text{P}_f v_h=0 ~\text{and } \text{P}_f w=0\} 
\end{align*}
Then, for $0<\varepsilon,\zeta \ll1 $ small enough satisfying the above assumptions the following statements hold
\begin{enumerate}[(i)]
    \item There exists a locally forward invariant manifold $S_{\varepsilon,\zeta}$ for the system (3.2) called the slow manifold, given by
\begin{align}
    S_{\varepsilon,\zeta}:=\{(h^{\varepsilon,\zeta}(v_h,w),v_h,w):~ v_h,w\in H^2(\Omega)\cap L^2_f(\Omega)\},
\end{align}
where the function $h^{\varepsilon,\zeta}$ is obtained via a Lyapunov-Perron fixed point argument;
\item The slow manifold $S_{\varepsilon,\zeta}$ has a distance of $\mathcal{O}(\varepsilon)$, measured in the $H^1$-norm, to the critical manifold $S_{0,\zeta}$;
\item The slow manifold is exponentially attracting and $C^1$-regular.
\item The semi-flow of solutions of \eqref{eq: gspt epsilon} on $S_{\varepsilon,\zeta}$ converges to the semi-flow of solutions of \eqref{eq: gspt limit} on $S_{0,\zeta}$ with respect to the $H^1$-norm.
\end{enumerate}

\end{theorem}

With this we can formulate the slow dynamics on the slow manifold $S_{\varepsilon,\zeta}$, similar to the one of the critical manifold, given by
\begin{align}\label{eq: slow system}
    \begin{split}
        \partial_t v_h^s&= d_{v_h}\Delta v_h^s + 2\tilde{p}_1 \text{P}_s\big( wh^{\varepsilon,\zeta}(v_h^s,w^s)\big) -  m_h v_h^s -p_2 v_h^s ,\\
        \partial_t w^s&= d_{w}\Delta w^s- \tilde{p}_1 \text{P}_s \big(wh^{\varepsilon,\zeta}(v_h^s,w^s)\big) -m_w w^s+ a,\\
        \partial_n v_h^s&=0,\quad \partial_n w^s=0,\\
        v_h^s(0)&=  \text{P}_s v_{h,0},\quad w^s(0)=  \text{P}_s w_0,
    \end{split}
\end{align}
where we write $v_h^s= \text{P}_s v_h$ and $w^s= \text{P}_s w$.\\


    

To conclude this section, we put the results into context.
The key takeaway of Lemma \ref{lem: convergence of sol} is that we can quantify the rate of convergence of the general system \eqref{eq: gspt epsilon} to the reduced system satisfying the Holling type II functional response, i.e., making the heuristic arguments of Section \ref{sec: formal reduction} precise.
However, we not only have the convergence of solutions but thanks to Theorem \ref{thm: gspt example} we also have knowledge of the dynamical properties of solutions, i.e., we can reduce the full system to a slow evolution on the slow manifold, which has a similar structure as the critical manifold in the reduced system.
And moreover, the semi-flow of solutions of the slow system \eqref{eq: slow system} converges to the semi-flow of the reduced system \eqref{eq: gspt limit}.
As this result works in both ways system \eqref{eq: gspt epsilon} and its slow subsystem \eqref{eq: slow system} can also be used to study properties of the reduced system \eqref{eq: gspt limit}.

\begin{remark}
    The GSPT can also be applied to the other reductions leading to Holling type I and III type functional responses.
\end{remark}

\section{Conclusion and Discussion}

In this work, we derived, from a free energy functional, a general class of closed multi-species reaction-diffusion systems using the Energetic Variational Approach (EnVarA). 
This derivation ensures thermodynamic consistency and provides a unified framework encompassing a wide range of models found in biological and ecological contexts. 
By reducing the closed system to an open subsystem and then performing systematic scaling limits, we obtained reduced systems that exhibit Holling type I, II, and III functional responses. 
In particular, we rigorously analyzed the reduction to Holling type II dynamics using the geometric singular perturbation theory (GSPT) for PDEs, which allowed us to demonstrate both the convergence of solutions and the persistence of dynamical features via the existence of slow invariant manifolds.\\

We conclude with several observations that point to future directions and broader implications of our work:
\\
Our current analysis considers a sequential reduction where the small parameters $\varepsilon_i$ for $i=1,\dots,7$ are simultaneously taken to zero, followed by specific scalings of the reaction rate $p_2$ to obtain the different functional responses. 
A natural extension would be to investigate the interplay between these limits when taken in different orders, or to study the full double limit problem directly. 
Such an approach could uncover additional reduced models or help clarify the robustness of functional responses under perturbations. 
This idea aligns with recent work on double limits \cite{kuehn2022general} in differential equations and remains a promising avenue for further research.\\

While this work focuses on Holling type I–III functional responses, it is worth noting that the DeAngelis–Beddington functional response, proposed independently by DeAngelis \cite{deangelis1975model} and Beddington \cite{beddington1975mutual}, offers an alternative that accounts for mutual interference among predators. 
This response has the form $f(x,y)= \frac{ax}{1+cx+by}$, resembling the Holling type II response with an additional coupling term. 
It can be derived from a mechanistic perspective by modeling prey heterogeneity, such as distinguishing between foraging and hiding individuals \cite{geritz2012mechanistic}. 
Due to the flexibility of our energetic framework, it is feasible to incorporate such behavior by appropriately modifying the reaction equations. 
What remains is is to verify the necessary assumptions for applying GSPT in this setting and to rigorously derive the corresponding reduced dynamics.\\

The reduced systems obtained in Section \ref{sec: formal reduction} are known to exhibit various spatial patterns, as shown in prior studies \cite{sun2010spatial,sun2012pattern,guin2015spatial,wang2017spatiotemporal}.
An important open question is whether these patterns persist in the full (perturbed) open system before the reduction step.
Our application of the generalized GSPT, particularly Theorem \ref{thm: gspt example}, offers a first step in rigorously connecting the dynamics of the reduced models with their perturbed counterparts.

\subsection*{Acknowledgement}
The author would like to thank C. Liu for the fruitful discussions and the valuable feedback.

\small{
\bibliographystyle{siam}
\bibliography{lit}}

\begin{thebibliography}{10}

\bibitem{arditi2012species}
{\sc R.~Arditi and L.~R. Ginzburg}, {\em How species interact: altering the
  standard view on trophic ecology}, Oxford University Press, 2012.

\bibitem{beddington1975mutual}
{\sc J.~R. Beddington}, {\em Mutual interference between parasites or predators
  and its effect on searching efficiency}, The Journal of Animal Ecology,
  (1975), pp.~331--340.

\bibitem{beveridge1989free}
{\sc D.~L. Beveridge and F.~M. Dicapua}, {\em Free energy via molecular
  simulation: applications to chemical and biomolecular systems}, Annual review
  of biophysics and biophysical chemistry, 18 (1989), pp.~431--492.

\bibitem{borgogno2009mathematical}
{\sc F.~Borgogno, P.~D'odorico, F.~Laio, and L.~Ridolfi}, {\em Mathematical
  models of vegetation pattern formation in ecohydrology}, Reviews of
  geophysics, 47 (2009).

\bibitem{britton1986reaction}
{\sc N.~F. Britton}, {\em Reaction-diffusion equations and their applications
  to biology.}, Academic Press, London, 1986.

\bibitem{cassani2021belousov}
{\sc A.~Cassani, A.~Monteverde, and M.~Piumetti}, {\em Belousov-zhabotinsky
  type reactions: The non-linear behavior of chemical systems}, Journal of
  mathematical chemistry, 59 (2021), pp.~792--826.

\bibitem{colombo2021first}
{\sc M.~Colombo and C.~Wright}, {\em First principles in the life sciences: the
  free-energy principle, organicism, and mechanism}, Synthese, 198 (2021),
  pp.~3463--3488.

\bibitem{conforto2018reaction}
{\sc F.~Conforto, L.~Desvillettes, and C.~Soresina}, {\em About
  reaction--diffusion systems involving the holling-type ii and the
  beddington--deangelis functional responses for predator--prey models},
  Nonlinear Differential Equations and Applications NoDEA, 25 (2018), p.~24.

\bibitem{cosner2008reaction}
{\sc C.~Cosner}, {\em Reaction--diffusion equations and ecological modeling},
  in Tutorials in Mathematical Biosciences IV: Evolution and Ecology, Springer,
  2008, pp.~77--115.

\bibitem{dawes2013derivation}
{\sc J.~Dawes and M.~Souza}, {\em A derivation of holling's type i, ii and iii
  functional responses in predator--prey systems}, Journal of theoretical
  biology, 327 (2013), pp.~11--22.

\bibitem{de2019non}
{\sc F.~De~Anna and C.~Liu}, {\em Non-isothermal general ericksen--leslie
  system: derivation, analysis and thermodynamic consistency}, Archive for
  Rational Mechanics and Analysis, 231 (2019), pp.~637--717.

\bibitem{de1936thermodynamic}
{\sc T.~De~Donder and P.~Van~Rysselberghe}, {\em Thermodynamic theory of
  affinity}, Stanford university press, 1936.

\bibitem{deangelis1975model}
{\sc D.~L. DeAngelis, R.~Goldstein, and R.~V. O'Neill}, {\em A model for tropic
  interaction}, Ecology, 56 (1975), pp.~881--892.

\bibitem{doelman1997pattern}
{\sc A.~Doelman, T.~J. Kaper, and P.~A. Zegeling}, {\em Pattern formation in
  the one-dimensional gray-scott model}, Nonlinearity, 10 (1997), p.~523.

\bibitem{ensing2006metadynamics}
{\sc B.~Ensing, M.~De~Vivo, Z.~Liu, P.~Moore, and M.~L. Klein}, {\em
  Metadynamics as a tool for exploring free energy landscapes of chemical
  reactions}, Accounts of chemical research, 39 (2006), pp.~73--81.

\bibitem{epstein1998introduction}
{\sc I.~R. Epstein and J.~A. Pojman}, {\em An introduction to nonlinear
  chemical dynamics: oscillations, waves, patterns, and chaos}, Oxford
  University Press, 1998.

\bibitem{evans2010partial}
{\sc L.~C. Evans}, {\em Partial differential equations}, vol.~19 of Grad. Stud.
  Math., American Mathematical Society, Providence, RI, 2nd~ed., 2010.

\bibitem{fenichel1979geometric}
{\sc N.~Fenichel}, {\em Geometric singular perturbation theory for ordinary
  differential equations}, Journal of Differential Equations, 31 (1979),
  pp.~53--98.

\bibitem{fisher1937wave}
{\sc R.~A. Fisher}, {\em The wave of advance of advantageous genes}, Annals of
  eugenics, 7 (1937), pp.~355--369.

\bibitem{ge2013dissipation}
{\sc H.~Ge and H.~Qian}, {\em Dissipation, generalized free energy, and a
  self-consistent nonequilibrium thermodynamics of chemically driven open
  subsystems}, Physical Review E—Statistical, Nonlinear, and Soft Matter
  Physics, 87 (2013), p.~062125.

\bibitem{ge2016mesoscopic}
\leavevmode\vrule height 2pt depth -1.6pt width 23pt, {\em Mesoscopic kinetic
  basis of macroscopic chemical thermodynamics: A mathematical theory},
  Physical Review E, 94 (2016), p.~052150.

\bibitem{ge2017mathematical}
\leavevmode\vrule height 2pt depth -1.6pt width 23pt, {\em Mathematical
  formalism of nonequilibrium thermodynamics for nonlinear chemical reaction
  systems with general rate law}, Journal of Statistical Physics, 166 (2017),
  pp.~190--209.

\bibitem{geritz2012mechanistic}
{\sc S.~Geritz and M.~Gyllenberg}, {\em A mechanistic derivation of the
  deangelis--beddington functional response}, Journal of theoretical biology,
  314 (2012), pp.~106--108.

\bibitem{gierer1972theory}
{\sc A.~Gierer and H.~Meinhardt}, {\em A theory of biological pattern
  formation}, Kybernetik, 12 (1972), pp.~30--39.

\bibitem{giga2017variational}
{\sc M.-H. Giga, A.~Kirshtein, and C.~Liu}, {\em Variational modeling and
  complex fluids}, Handbook of mathematical analysis in mechanics of viscous
  fluids,  (2017), pp.~1--41.

\bibitem{guin2015spatial}
{\sc L.~N. Guin}, {\em Spatial patterns through turing instability in a
  reaction--diffusion predator--prey model}, Mathematics and Computers in
  Simulation, 109 (2015), pp.~174--185.

\bibitem{hao2024pattern}
{\sc W.~Hao, C.~Liu, Y.~Wang, and Y.~Yang}, {\em On pattern formation in the
  thermodynamically-consistent variational gray-scott model}, arXiv preprint
  arXiv:2409.04663,  (2024).

\bibitem{holling1959components}
{\sc C.~S. Holling}, {\em The components of predation as revealed by a study of
  small-mammal predation of the european pine sawfly}, The canadian
  entomologist, 91 (1959), pp.~293--320.

\bibitem{holling1959some}
\leavevmode\vrule height 2pt depth -1.6pt width 23pt, {\em Some characteristics
  of simple types of predation and parasitism}, The canadian entomologist, 91
  (1959), pp.~385--398.

\bibitem{holmes1994partial}
{\sc E.~E. Holmes, M.~A. Lewis, J.~Banks, and R.~Veit}, {\em Partial
  differential equations in ecology: spatial interactions and population
  dynamics}, Ecology, 75 (1994), pp.~17--29.

\bibitem{hsieh2020global}
{\sc C.-Y. Hsieh, T.-C. Lin, C.~Liu, and P.~Liu}, {\em Global existence of the
  non-isothermal poisson--nernst--planck--fourier system}, Journal of
  Differential Equations, 269 (2020), pp.~7287--7310.

\bibitem{hu2008free}
{\sc H.~Hu and W.~Yang}, {\em Free energies of chemical reactions in solution
  and in enzymes with ab initio quantum mechanics/molecular mechanics methods},
  Annu. Rev. Phys. Chem., 59 (2008), pp.~573--601.

\bibitem{hummel2022slow}
{\sc F.~Hummel and C.~Kuehn}, {\em Slow manifolds for infinite-dimensional
  evolution equations}, Commentarii Mathematici Helvetici, 97 (2022),
  pp.~61--132.

\bibitem{Jones1995}
{\sc C.~K. R.~T. Jones}, {\em Geometric singular perturbation theory},
  Springer, Berlin Heidelberg, 1995, pp.~44--118.

\bibitem{jorgensen2004towards}
{\sc S.~E. Jorgensen and Y.~M. Svirezhev}, {\em Towards a thermodynamic theory
  for ecological systems}, Elsevier, 2004.

\bibitem{klausmeier1999regular}
{\sc C.~A. Klausmeier}, {\em Regular and irregular patterns in semiarid
  vegetation}, Science, 284 (1999), pp.~1826--1828.

\bibitem{koch1994biological}
{\sc A.-J. Koch and H.~Meinhardt}, {\em Biological pattern formation: from
  basic mechanisms to complex structures}, Reviews of modern physics, 66
  (1994), p.~1481.

\bibitem{kolmogorov1937study}
{\sc A.~Kolmogorov, I.~Petrovskii, and N.~Piskunov}, {\em A study of the
  equation of diffusion with increase in the quantity of matter, and its
  application to a biological problem}, Moscow University Bulletin of
  Mathematics, 1 (1937), pp.~1--25.

\bibitem{kondepudi2014modern}
{\sc D.~Kondepudi and I.~Prigogine}, {\em Modern thermodynamics: from heat
  engines to dissipative structures}, John Wiley \& Sons, New York, 2014.

\bibitem{kondo2010reaction}
{\sc S.~Kondo and T.~Miura}, {\em Reaction-diffusion model as a framework for
  understanding biological pattern formation}, science, 329 (2010),
  pp.~1616--1620.

\bibitem{kuehn2015multiple}
{\sc C.~Kuehn}, {\em Multiple time scale dynamics}, vol.~191, Springer, Cham,
  2015.

\bibitem{kuehn2022general}
{\sc C.~Kuehn, N.~Berglund, C.~Bick, M.~Engel, T.~Hurth, A.~Iuorio, and
  C.~Soresina}, {\em A general view on double limits in differential
  equations}, Physica D: Nonlinear Phenomena, 431 (2022), p.~133105.

\bibitem{kuehn2024infinite}
{\sc C.~Kuehn, P.~Lehner, and J.-E. Sulzbach}, {\em Infinite dimensional slow
  manifolds for a linear fast-reaction system}, in Topics in Multiple Time
  Scale Dynamics, vol.~806 of Contemp. Math., Amer. Math. Soc., Providence, RI,
  2024.

\bibitem{kuehn2025approximate}
{\sc C.~Kuehn and J.-E. Sulzbach}, {\em Approximate slow manifolds in the
  fokker-planck equation}, arXiv preprint arXiv:2501.18981,  (2025).

\bibitem{kuehn2025fast}
\leavevmode\vrule height 2pt depth -1.6pt width 23pt, {\em Fast reactions and
  slow manifolds}, Nonlinear Differential Equations and Applications NoDEA, 32
  (2025), pp.~1--43.

\bibitem{kuttler2017reaction}
{\sc C.~Kuttler}, {\em Reaction--diffusion equations and their application on
  bacterial communication}, in Handbook of Statistics, vol.~37, Elsevier, 2017,
  pp.~55--91.

\bibitem{liang2022reversible}
{\sc J.~Liang, N.~Jiang, C.~Liu, Y.~Wang, and T.-F. Zhang}, {\em On a
  reversible gray-scott type system from energetic variational approach and its
  irreversible limit}, Journal of Differential Equations, 309 (2022),
  pp.~427--454.

\bibitem{liu2022brinkman}
{\sc C.~Liu and J.-E. Sulzbach}, {\em The brinkman-fourier system with ideal
  gas equilibrium}, Discrete and Continuous Dynamical Systems-Series A, 42
  (2022), pp.~425--462.

\bibitem{liu2022well}
\leavevmode\vrule height 2pt depth -1.6pt width 23pt, {\em Well-posedness for
  the reaction-diffusion equation with temperature in a critical besov space},
  Journal of Differential Equations, 325 (2022), pp.~119--149.

\bibitem{lotka1920analytical}
{\sc A.~J. Lotka}, {\em Analytical note on certain rhythmic relations in
  organic systems}, Proceedings of the National Academy of Sciences, 6 (1920),
  pp.~410--415.

\bibitem{meron2012pattern}
{\sc E.~Meron}, {\em Pattern-formation approach to modelling spatially extended
  ecosystems}, Ecological Modelling, 234 (2012), pp.~70--82.

\bibitem{murray2007mathematical}
{\sc J.~D. Murray}, {\em Mathematical biology: II. Spatial Models and
  Biomedical Applications}, Springer, Berlin, 1989.

\bibitem{okubo2001diffusion}
{\sc A.~Okubo, S.~A. Levin, et~al.}, {\em Diffusion and ecological problems:
  modern perspectives}, vol.~14, Springer, 2001.

\bibitem{oster1974chemical}
{\sc G.~F. Oster and A.~S. Perelson}, {\em Chemical reaction dynamics: Part i:
  Geometrical structure}, Archive for Rational Mechanics and Analysis, 55
  (1974), pp.~230--274.

\bibitem{perelson1974chemical}
{\sc A.~S. Perelson and G.~F. Oster}, {\em Chemical reaction dynamics part ii:
  Reaction networks}, Archive for Rational Mechanics and Analysis, 57 (1974),
  pp.~31--98.

\bibitem{pierre2010global}
{\sc M.~Pierre}, {\em Global existence in reaction-diffusion systems with
  control of mass: a survey}, Milan Journal of Mathematics, 78 (2010),
  pp.~417--455.

\bibitem{real1979ecological}
{\sc L.~A. Real}, {\em Ecological determinants of functional response},
  Ecology, 60 (1979), pp.~481--485.

\bibitem{rietkerk2002self}
{\sc M.~Rietkerk, M.~C. Boerlijst, F.~Van~Langevelde, R.~HilleRisLambers, J.~v.
  de~Koppel, L.~Kumar, H.~H. Prins, and A.~M. de~Roos}, {\em Self-organization
  of vegetation in arid ecosystems}, The American Naturalist, 160 (2002),
  pp.~524--530.

\bibitem{rietkerk2008regular}
{\sc M.~Rietkerk and J.~Van~de Koppel}, {\em Regular pattern formation in real
  ecosystems}, Trends in ecology \& evolution, 23 (2008), pp.~169--175.

\bibitem{segel1989quasi}
{\sc L.~A. Segel and M.~Slemrod}, {\em The quasi-steady-state assumption: a
  case study in perturbation}, SIAM review, 31 (1989), pp.~446--477.

\bibitem{shanks2001modeling}
{\sc N.~Shanks}, {\em Modeling biological systems: the belousov--zhabotinsky
  reaction}, Foundations of Chemistry, 3 (2001), pp.~33--53.

\bibitem{sherratt1997oscillations}
{\sc J.~A. Sherratt, B.~T. Eagan, and M.~A. Lewis}, {\em Oscillations and chaos
  behind predator--prey invasion: mathematical artifact or ecological
  reality?}, Philosophical transactions of the Royal Society of London. Series
  B: Biological Sciences, 352 (1997), pp.~21--38.

\bibitem{shigesada1997biological}
{\sc N.~Shigesada}, {\em Biological Invasions: Theory and Practice}, vol.~205,
  Oxford University Press, 1997.

\bibitem{solomon1949natural}
{\sc M.~E. Solomon}, {\em The natural control of animal populations}, The
  Journal of Animal Ecology,  (1949), pp.~1--35.

\bibitem{sun2010spatial}
{\sc G.~Sun, S.~Sarwardi, P.~Pal, and M.~S. Rahman}, {\em The spatial patterns
  through diffusion-driven instability in modified leslie-gower and
  holling-type ii predator-prey model}, Journal of Biological Systems, 18
  (2010), pp.~593--603.

\bibitem{sun2012pattern}
{\sc G.-Q. Sun, J.~Zhang, L.-P. Song, Z.~Jin, and B.-L. Li}, {\em Pattern
  formation of a spatial predator--prey system}, Applied Mathematics and
  Computation, 218 (2012), pp.~11151--11162.

\bibitem{tang2023rigorous}
{\sc B.~Q. Tang and B.-N. Tran}, {\em Rigorous derivation of michaelis-menten
  kinetics in the presence of diffusion}, arXiv preprint arXiv:2303.07913,
  (2023).

\bibitem{turing1990chemical}
{\sc A.~M. Turing}, {\em The chemical basis of morphogenesis}, Bulletin of
  mathematical biology, 52 (1990), pp.~153--197.

\bibitem{volterra1926fluctuations}
{\sc V.~Volterra}, {\em Fluctuations in the abundance of a species considered
  mathematically}, Nature, 118 (1926), pp.~558--560.

\bibitem{wang2017spatiotemporal}
{\sc J.~Wang}, {\em Spatiotemporal patterns of a homogeneous diffusive
  predator--prey system with holling type iii functional response}, Journal of
  Dynamics and Differential Equations, 29 (2017), pp.~1383--1409.

\bibitem{wang2022some}
{\sc Y.~Wang and C.~Liu}, {\em Some recent advances in energetic variational
  approaches}, Entropy, 24 (2022), p.~721.

\bibitem{wang2020field}
{\sc Y.~Wang, C.~Liu, P.~Liu, and B.~Eisenberg}, {\em Field theory of
  reaction-diffusion: Law of mass action with an energetic variational
  approach}, Physical Review E, 102 (2020), p.~062147.

\end{thebibliography}

\end{document}